\documentclass[11pt]{amsart}
\usepackage{wasysym,amssymb,eufrak,indentfirst,graphicx,cite,amsthm,color}
\usepackage[bookmarksnumbered, colorlinks, linktocpage, plainpages]{hyperref}
\usepackage{epsfig}
\usepackage{float}
\usepackage{mathdots}
\usepackage{extarrows}
\usepackage{indentfirst}
\newtheorem{theorem}{Theorem}[section]

\newtheorem{lemma}[theorem]{Lemma}

\theoremstyle{definition}

\theoremstyle{thma}
\newtheorem*{thma}{Theorem A}

\theoremstyle{remark}

\newtheorem{example}[theorem]{Example}
\numberwithin{equation}{section}
\begin{document}

\title[Integrability of distortion of inverses]
{Planar mappings of subexponentially integrable distortion: integrability of distortion of inverses}

\author[H.-Q. Xu]{Haiqing Xu}
\address{School of Mathematical Sciences, University of Science and
Technology of China, Hefei 230026, China}

\email{hqxu$\symbol{64}$mail.ustc.edu.cn}

\thanks{This work was supported by the National Natural Science
Foundation of China grants 11571333, 11471301.}


\begin{abstract}
We establish the optimal regularity for the distortion of inverses of mappings of finite distortion with logarithm-iterated style subexponentially integrable distortion, which generalizes the Theorem 1. of [J. Gill, Ann. Acad. Sci. Fenn. Math. 35 (2010), no. 1, 197--207].
\end{abstract}

\maketitle

\section{Introduction}
We say that a mapping $f:\Omega \rightarrow \mathbb{R}^{n}$ in a domain $\Omega \subset \mathbb{R}^{n}$ is a mapping of finite distortion, if
\begin{enumerate}
\item[(i)] $\ f \in W^{1,1}_{loc}(\Omega,\mathbb{R}^{n}),$
\item[(ii)] the Jacobian determinant $J_{f}(z) \in  L^{1}_{loc}(\Omega)$ and
\item[(iii)] there is a measurable function $K(z):\Omega \rightarrow [1,+\infty]$ with $K(z)\ <\ \infty$ almost everywhere such that
\begin{equation}\label{distortion inequality}
|Df(z)|^{n}\ \le \  K(z)J_{f}(z)\ \ \ \mbox{for\ \ almost\ all }\ z \in \Omega,
\end{equation}
\end{enumerate}
where $|Df(z)|$ is the operator norm of matrix $Df(z)$.
For mappings of finite distortion, we define the distortion function by
\begin{equation*}
K_{f}(z)\ =\
\begin{cases}
\frac{|Df(z)|^n}{J_{f}(z)},  &\mbox{if } z \in  \left\{z \in \Omega:J_{f}(z)\ > \ 0 \right\} \\
1, & \mbox{if } z \in  \left\{z \in \Omega:J_{f}(z)\ =\ 0 \right\},
\end{cases}
\end{equation*}
then the distortion inequality (\ref{distortion inequality}) becomes
\begin{equation}\label{distortion equality}
|Df(z)|^{n}\ = \  K_f(z)J_{f}(z).
\end{equation}
We will limit the discussion in this paper to the planar case, i.e. $n=2.$ In this case, since $|Df(z)|=|f_z|+|f_{\bar{z}}|$ and $J_f(z)=|f_z|^2-|f_{\bar{z}}|^2$, the distortion equality (\ref{distortion equality}) is equivalent to the Beltrami equation
\begin{equation}\label{beltrami equality}
\frac{\partial{f(z)}}{\partial{\bar{z}}} = \mu(z)\frac{\partial{f(z)}}{\partial z}
\end{equation}
where $ \frac{\partial}{\partial \bar{z}} = \frac{1}{2}\left( \frac{\partial}{\partial x} + i\frac{\partial}{\partial y}\right)$, $ \frac{\partial}{\partial z} = \frac{1}{2}\left( \frac{\partial}{\partial x} - i\frac{\partial}{\partial y}\right)$ and $|\mu(z)| = \frac{K_f(z) - 1}{K_{f}(z) + 1}.$
For more details about mappings of finite distortion, we refer the reader to \cite{hencl2014lectures} and the references therein.

If $||\mu||_{\infty}\ \le\ k\ <1$, then the classical measurable Riemann mapping theorem tells that the Beltrami equation (\ref{beltrami equality}) admits a homeomorphic solution and other solutions are represented by composing the homeomorphic solution with holomorphic functions, see \cite{ahlfors2006lectures,astala2009elliptic}.

When $||\mu||_{\infty} = 1$, the Beltrami equation (\ref{beltrami equality}) becomes degenerate. David dealt with this degenerate Beltrami equation in \cite{david1988solution}, where he generalized the measurable Riemann mapping theorem when the distortion function $K$ satisfies $\exp(pK) \in L^1_{loc}(\Omega)$ for some $p\ >\ 0$. And David also noted that it is not necessary that the distortion of $f^{-1}$ is exponentially integrable when $\exp(K_{f})$ is integrable. Later, Hencl and Koskela \cite{hencl2006regularity} proved $K_{f^{-1}} \in L^{\beta}_{loc}$ where $\beta\ =\ c_{0}p$ with absolute constant $c_{0}$, under the local integrability of $\exp(p K_{f})$. Based on Theorem 1.1 of \cite{astala2010optimal}, Gill \cite{gillJMAA} ascertained the sharp inequality $c_{0}\ < \ 1$.
The comprehensive statement is as follows.

\begin{thma}
Suppose $f :\Omega \rightarrow \mathbb{R}^2$ is a homeomorphic mapping of finite distortion to the Beltrami equation (\ref{beltrami equality}), with the associated distortion function $K_{f}$. If
\begin{equation*}
\exp(pK_{f}) \in L^{1}_{loc}(\Omega),
\end{equation*}
then $f^{-1}$ is a mapping of finite distortion and the distortion function $K_{f^{-1}}$ satisfies
\begin{equation*}
K_{f^{-1}} \in L^{\beta}_{loc}(f(\Omega))\ \  \mbox{   for }\ 0\ <\ \beta\ <\ p.
\end{equation*}
Moreover this result is sharp in the sense that for every $p\ >\ 0$ there are functions $f$ as above such that $K_{f^{-1}} \notin L_{loc}^{p}$.
\end{thma}

Let
\begin{equation}\label{A_{p,n}}
\mathcal{A}_{p,n}(x)\ =\ \frac{px}{1\ +\ \log_{(1)}(x) \log_{(2)}(e\ -\ 1\ +\ x) \cdots \log_{(n)}\left(e^{e^{\iddots^{e}}}\ -\ 1 +\ x\right)}\ - \ p,
\end{equation}
where
$\log_{(i)}(x)\ =\ \log \left(\cdots \left(\log \left (\log(x)\right) \right)\cdots \right)$ and $\exp_{(i)}(x)\ =\ \exp  \left(\cdots \left( \exp \left ( \exp(x)\right) \right)\cdots \right)$ are $i$-iterated logarithm and exponent for $i\ =\ 1,\ 2,\ \cdots$.
Gill in \cite{gill2010planar} generalized Theorem 1.1. of \cite{astala2010optimal} to the solution $f$ to the Beltrami equation (\ref{beltrami equality}) with $\exp\left[\mathcal{A}_{p,n}(K_{f}) \right] \in L^{1}_{loc}$.
However there is no corresponding result analogous to Theorem A..

The aim of this article is to present a generalization of Theorem A. under the local integrability of $\exp[\mathcal{A}_{p,n}(K_f)]$.

\begin{theorem}\label{main theorem}
Suppose $f :\Omega \rightarrow \mathbb{R}^2$ is a homeomorphic mapping of finite distortion to the Beltrami equation (\ref{beltrami equality}), with the associated distortion function $K_{f}$. If
\begin{equation}\label{main condition}
\exp\left[\mathcal{A}_{p,n}(K_{f}) \right] \in L^{1}_{loc}(\Omega),
\end{equation}
then $f^{-1}$ is a mapping of finite distortion and the distortion function $K_{f^{-1}}$ satisfies
\begin{equation}\label{main result}
\log_{(n)}\left(e^{e^{\iddots^{e}}}\ +\ K_{f^{-1}}\right) \in L^{\beta}_{loc}(f(\Omega))\ \  \mbox{   for }\ 0\ <\ \beta\ <\ p.
\end{equation}
Moreover, for every $p\ >\ 0$ there are mappings that satisfy the assumption of theorem, yet fail (\ref{main result}) for $\beta\ =\ p$.
\end{theorem}
The $e^{e^{\iddots^{e}}}$ in functions $\log_{(i)}\left(e^{e^{\iddots^{e}}}\ +\ x\right)$ means $\exp_{(i-1)}(e)$ if there is not special announcement.

The rest of the paper is organized as follows. In section 2, we recall some basic facts about Legendre Transformation and obtain an inequality of Young type. The section 3 is devoted to the proof of Theorem \ref{main theorem}.

\subsection*{Notation}
$s\ \gg \ 1$ denotes $s$ is sufficiently large and $s\ \ll \ 1$ denotes $s$ is sufficiently small.
$f(x)\ \lesssim \ g(x)$ and $f(y)\ \gtrsim \ g(y)$ mean that there exist constants $M$ and $m$ such that $f(x)\ \le \ Mg(x)$ and $f(y)\ \ge \ mg(y)$ for suitable $x$ and $y$. $f(x)\ \sim \ g(x)$ means $f(x)\ \lesssim \ g(x)$ and $f(x)\ \gtrsim \ g(x)$.
When concerned only with the convergence of improper integrals, we use notations $\int_{*}^{\infty}$ and $\int_{0}^{*}.$

\section{An inequality of Young type}
We begin by recalling some basic facts about Legendre Transformation from \cite{gtm60}.

Suppose function $\Phi(t)$ is convex and $\Phi''(t)\ > \ 0$ for $t\ \ge \ 0$, the Legendre Transformation of $\Phi(t)$ is
$$\Phi ^{*}(s) \ = \ \underset{ t \ge 0}{\mbox{max}} \left\{ s t \ - \ \Phi(t)\right\}\ \ \mbox{for}\ \ s\ \ge \ 0.$$
Directly from the definition, we obtain
\begin{equation}\label{from definition}
t s\ \le\ \Phi(t) \ + \ \Phi^{*}(s)\ \ \mbox{ for }\ t,\ s\ \ge\ 0.
\end{equation}

\begin{lemma}\label{remark 1}
$(\Phi^{*})'(s)\ = \ (\Phi')^{-1}(s).$
\end{lemma}
\begin{proof}
Given $s\ \ge \ 0,$ let $t(s)$ be the value such that the maximal of $ st\ - \ \Phi(t)$ is obtained, i.e.
$$\Phi^{*}(s)\ =\ s t(s)\ -\ \Phi(t(s)),$$
so $\Phi'(t(s))\ =\ s$, then $t(s)\ =\ (\Phi')^{-1}(s)$.
Consequently, we have
\begin{align*}
(\Phi^{*})'(s)&\ =\ t(s)\ +\ (s\ -\ \Phi'(t(s)))\frac{d(\Phi')^{-1}(s)}{ds} \\
&\ =\ (\Phi')^{-1}(s).\\
\end{align*}
\end{proof}

Given a strictly convex $C^2$ function $\Phi(t)$, it is not easy to compute the explicit expression of $\Phi^{*}(s)$ from the definition. However, by Lemma \ref{remark 1}, we can obtain the asymptotical behaviour of $\Phi^{*}(s)$ as $s\  \gg\ 1$. The following example, coming from \cite{giannetti2010the}, illustrates this.

\begin{example}\label{example}
Put $\Phi(t)\ =\ \exp\left( \frac{t}{\log(e\ +\ t)}\right)$.
After differentiating and taking the logarithm,
we have
$$\log (\Phi'(t)) \ \sim \ \frac{t}{\log (t)}\ \ \mbox{ as }\  t \gg 1.$$
Let $\frac{t}{\log (t)} \ = \ \log(s)$, then
\begin{equation}\label{example-1}
t\ \sim \ \log(s) \log_{(2)}(s) \ \ \mbox{ as }\ s\ \gg \ 1.
\end{equation}
In other words, when $t$ satisfies (\ref{example-1}), we have
$$\log (\Phi'(t)) \ \sim \ \log(s) \ = \  \log[\Phi'((\Phi^{*})'(s))].$$
By the monotonicity of $\log(\cdot)$ and $\Phi'(\cdot)$, we have
$$(\Phi^{*})'(s)\ \sim\ \log(s) \log_{(2)}(s).$$
Hence, by the Newton-Leibniz formula, we show
$$\Phi^{*}(s)\ \sim \ s \log(s) \log_{(2)}(s)\ \ \mbox{ as } \ s \ \gg \ 1.$$
\end{example}

By the method analogous to Example \ref{example}, we present an inequality of Young type, which plays the crucial role in the proof of Theorem \ref{main theorem}.

\begin{lemma}\label{Young}
Given $\beta \ > \ 0$, there exist constants $C_{1},\ C_{2}\ >\ 0 $ such that
\begin{equation*}
ts\ \le \ C_{1} \Phi(t)\ +\ C_{2}\Psi(s) \ \ \mbox{ for }\ t,\ s \ \ge \ 0,
\end{equation*}
where $\Phi(t)\ =\ \exp[\mathcal{A}_{p,n}(\exp_{(n)}(t^{\frac{1}{\beta}}))]$ and
$\Psi(s)\ =\ s\left[\log_{(n+1)}\left(e^{e^{\iddots^{e}}}\ +\ s\right)\right]^{\beta}.$
\end{lemma}

\begin{proof}
We divide the proof into two cases.

\bigskip
  \item [\textsf{Case} 1:]
$0\ \le\ s\ \le \ C_{1}$ for some $C_1 \ >\ 0$.

Since $t\ \le \ \Phi(t)$ for $t\ \ge \ 0$ and $\Psi(s) \ \ge \ 0$ for $s\ \ge \ 0$, we obtain

\begin{equation}\label{Young 1}
st\ \le\  C_1 \Phi (t)\ +\ \Psi(s)\ \ \mbox{ for }\ t\ \ge \ 0   \ \mbox{ and }\ 0\ \le \ s\ \le \ C_{1}.
\end{equation}

\bigskip
  \item [\textsf{Case} 2:]
$s \ \gg \ 1$.

Since
$$\log \Phi'(t) \ \sim \ \frac{ \exp_{(n)}(t^{\frac{1}{\beta}})}{ \exp_{(n-1)}(t^{\frac{1}{\beta}}) \cdots  \exp_{(1)}(t^{\frac{1}{\beta}}) t^{\frac{1}{\beta}}}\ \ \mbox{ for }\ t\ \gg \ 1$$
and
$$\exp_{(n-2)}(t^{\frac{1}{\beta}}) \cdots \exp_{(1)}(t^{\frac{1}{\beta}}) t^{\frac{1}{\beta}} \ <\ \exp_{(n-1)}(t^{\frac{1}{\beta}})\ \ \mbox{ for }\ t \ \ge \  0,$$
we have

\begin{equation}\label{Young 2}
 \frac{ \exp_{(n)}(t^{\frac{1}{\beta}})}{\exp_{(n-1)}(t^{\frac{1}{\beta}})} \ >\ \log \Phi'(t) \ > \ \frac{ \exp_{(n)}(t^{\frac{1}{\beta}})}{\left[ \exp_{(n-1)}(t^{\frac{1}{\beta}})\right]^{2}} \ \ \mbox{ for }\ t\ \gg \ 1.
\end{equation}

Next consider RHS of (\ref{Young 2}). Let $b\ =\ \exp_{(n)}(t^{\frac{1}{\beta}}),$ we consider
\begin{equation}\label{Young 2-0}
\frac{b}{[\log(b)]^{2}}\ =\ \log(s),\ \mbox{ i.e. }\ \frac{b^{\frac{1}{2}}}{\log({b^{\frac{1}{2}})}}\ =\ \sqrt{4\log(s)}.
\end{equation}
By Example \ref{example}, we have
$$b^{\frac{1}{2}} \ \sim \ \sqrt{\log(s)} \log_{(2)}(s) \ \ \mbox{ as } \ s \ \gg \ 1.$$
Taking $n$ successive logarithms,
we have
\begin{equation}\label{Young 2-1}
t \ \sim \ [\log_{(n+1)}(s)]^{\beta} \ \ \mbox{ as }\ s\ \gg \ 1.
\end{equation}
In other words, when $t$ satisfies (\ref{Young 2-1}), it follows from (\ref{Young 2-0}) and RHS of (\ref{Young 2}) that
$$\log(\Phi'(t))\ > \ \log(s)\ =\ \log[\Phi'((\Phi^{*})'(s))].$$
So, by the monotonicity of $\log(\cdot)$ and $\Phi'(\cdot)$, we have
\begin{equation}\label{Young 3}
(\Phi^{*})'(s)\ \lesssim \ [\log_{(n+1)}(s)]^{\beta} \ \ \mbox{ as }\ s \ \gg \ 1.
\end{equation}

For LHS of (\ref{Young 2}), by the argument similar to the one used in RHS of (\ref{Young 2}), we obtain
\begin{equation}\label{Young 4}
(\Phi^{*})'(s)\ \gtrsim \ [\log_{(n+1)}(s)]^{\beta} \ \ \mbox{ as }\ s\ \gg \ 1.
\end{equation}

Combining (\ref{Young 3}) and (\ref{Young 4}), we obtain
$$(\Phi^{*})'(s)\ \sim \ [\log_{(n+1)}(s)]^{\beta} \ \ \mbox{ as }\ s \ \gg \ 1.$$
Hence, by the Newton-Leibniz formula, we get
\begin{equation}\label{Young 4-1}
\Phi ^{*}(s)\ \sim \ s[\log_{(n+1)}(s)]^{\beta}\ <\ \Psi(s) \ \ \mbox{ as } \ s \ \gg \  1.
\end{equation}
It follows from (\ref{from definition}) and (\ref{Young 4-1}) that there exists constant $C_2 \ >\ 0$ such that
\begin{equation}\label{Young 5}
ts\ \le \ \Phi (t)\ +\ C_2 \Psi(s)\ \ \mbox{ for }\ t \ \ge \ 0\ \mbox{ and }\ s \ \gg \ 1.
\end{equation}

Combining (\ref{Young 1}) and (\ref{Young 5}), we complete the proof.
\end{proof}

\section{Proof of Theorem \ref{main theorem}}
We begin with four lemmas.

\begin{lemma}[\cite{kauhanen2003mapping}, Theorem 1.1.]\label{lusin N}
Suppose that $\Psi$ is a strictly increasing, differentiable function and satisfies
\begin{equation}\label{C-1}
\int_{1}^{\infty}\frac{\Psi'(t)}{t}dt \ =\ \infty , \tag{C\ --\ 1}
\end{equation}
\begin{equation}\label{C-2}
\lim\limits_{t \rightarrow \infty}t\Psi'(t)\ =\ \infty. \tag{C\ --\ 2}
\end{equation}
Let $f:\Omega \rightarrow \mathbb{R}^{n}$ be a mapping of finite distortion and the distortion function $K_f$ satisfies $\exp(\Psi(K_f))\in L^1_{loc}(\Omega)$. Then $f$ satisfies the Lusin's condition (N), i.e. $f(E)$ has Lebesgue measure zero if $E$ has Lebesgue measure zero.
\end{lemma}

Given a mapping $f:\Omega \rightarrow \mathbb{R}^{n}$, we denote $N(f, \Omega, y)$ by the number of preimages of point $y$ in $\Omega$ under $f$. We say $f$ has essentially bounded multiplicity, if $N(f, \Omega, y)$ is bounded for a.e. $y\ \in \ \mathbb{R}^n $.

From the proof of Theorem 1.2 in \cite{koskela2003mapping}, we know the assertion of Theorem 1.2 in \cite{koskela2003mapping} remains valid if both the mapping and its distortion function lie in local Sobolev spaces. So, we have the following result.

\begin{lemma}\label{Lusin N^-1}
Let $f: \Omega \rightarrow \mathbb{R}^2$ be a mapping of finite distortion and the distortion function $K_f$ satisfies $K_f \in L^1_{loc}(\Omega)$. If $f$ has essentially bounded multiplicity and $f$ is not a constant, then $J_{f} \ > \ 0$ almost everywhere in $\Omega$.
\end{lemma}

Suppose that a function $\mathcal{A}$ has the properties :
\begin{equation}\label{A-1}
\mathcal{A}:[1,\infty) \rightarrow [0,\infty) \mbox{ is a smooth increasing function with } \mathcal{A}(1)=0. \tag{A--1}
\end{equation}
\begin{equation}\label{A-2}
\int_{1}^{\infty} \frac{\mathcal{A}(t)}{t^2}\, \mathrm{d}t =\infty. \tag{A--2}
\end{equation}
The associated function of $\mathcal{A}$ is denoted by
\begin{equation}\label{P}
P(t)\ =\
\begin{cases}
t^{2},  & 0\ \le \ t\ \le \ 1 \\
\frac{t^{2}}{\mathcal{A}^{-1}(logt^{2})}, & t\ \ge \ 1.
\end{cases}
\end{equation}
Let us recall the notation
$$W^{1,Q}_{loc}\ = \ \left\{ f \in W^{1,1}_{loc}(\Omega)\ :\ Q(|Df|)\in L^{1}_{loc}(\Omega) \right\}.$$

\begin{lemma}[\cite{astala2009elliptic}, Theorem 20.5.1.]\label{thm20.5.1}
Given a function $\mathcal{A}$ satisfying (\ref{A-1}) and (\ref{A-2}) and the associated function $P$ is defined by (\ref{P}).
Let $f: \Omega \rightarrow \mathbb{R}^2$ be a mapping of finite distortion and the distortion function $K_{f}$ satisfies $exp[\mathcal{A}(K_{f})] \in L^{1}_{loc}(\Omega)$, then
$$f\ \in\ W^{1,P}_{loc}(\Omega).$$
\end{lemma}

Obviously, $\mathcal{A}_{p,n}$ satisfies (\ref{A-1}) and (\ref{A-2}). We denote the associated function of $\mathcal{A}_{p,n}$ by $P_n$.
Next we present a lemma essentially due to Gill \cite{gill2010planar}.

\begin{lemma}\label{proposition}
Suppose $f \in W^{1,P_n}_{loc}(\Omega)$ is a solution to the Beltrami equation (\ref{beltrami equality}) in a domain $\Omega \subset \mathbb{R}^2$ and the distortion function $K_{f}(z)$ satisfies $\exp[\mathcal{A}_{p,n}(K_{f}(z))] \in L^{1}_{loc}(\Omega)$,
then for all $0<\ \beta \ < \ p$, we have
\begin{equation*}
J_{f} \left[\log_{(n+1)}\left(e^{e^{\iddots^{e}}}\ +\ J_{f}\right)\right]^{\beta} \in L^{1}_{loc}(\Omega).
\end{equation*}
\end{lemma}

We now prove Theorem \ref{main theorem}.

\begin{proof}
Since
$$\mathcal{A}_{p,n}'(x)\ \gtrsim \ \frac{1}{\log_{(1)}(x)  \log_{(2)}(x)  \cdots  \log_{(n)}(x)}\ \ \mbox{ as }\ x \ \gg \ 1,$$
we know $\mathcal{A}_{p,n}(x)$ satisfies (\ref{C-1}) and (\ref{C-2}). It follows from Lemma \ref{lusin N} that $f$ satisfies the Lusin's condition (N).

Since
$$x\ \lesssim \ \exp(\mathcal{A}_{p,n}(x)) \ \ \mbox{ for }\ x \ \ge \ 1,$$
it follows from (\ref{main condition}) that
\begin{equation}\label{only one}
K_f \in L^1_{loc}(\Omega).
\end{equation}
So, Lemma \ref{Lusin N^-1} tells us $J_{f} \ > \ 0$ almost everywhere in $\Omega$.

Given compact set $\widetilde{M} \subset f(\Omega)$, we have $M= f^{-1}(\widetilde{M}) \subset \Omega$ is a compact set. By Corollary 3.3.3 in \cite{astala2009elliptic}, we obtain $f$ is differentiable almost everywhere in $\Omega$. So, we can divide the set $M$ into two subsets $M'$ and $M''$,
where $M'$ is the subset in which $f$ is differentiable and $ J_{f}(z)\ > \ 0$ and $M''\ =\ M  \setminus M'$ has Lebesgue measure zero. For any $z \in M'$, by Lemma A.29 of \cite{hencl2014lectures}, we have
$$Df^{-1}(f(z))\ =\ (Df(z))^{-1},$$
so $|Df^{-1}(f(z))|^2 J_f(z)\ =\ K_f(z)$ and $K_{f^{-1}}(f(z))\ =\ K_{f}(z)$. So, it follows from Corollary A.36 (c) of \cite{hencl2014lectures} and the Lusin's condition (N) of $f$ that

\begin{equation}\label{main theorem 1}
\int_{\widetilde{M}} |Df^{-1}(w)|^{2}\, \mathrm{d}w \ = \ \int_{M} K_{f}(z)\, \mathrm{d}z
\end{equation}
and
\begin{equation}\label{main theorem 2}
\int_{\widetilde{M}}\left[ \log_{(n)}\left(e^{e^{\iddots^{e}}}\ +\ K_{f^{-1}}\right) \right]^{\beta}\, \mathrm{d}w\ =\ \int_{M}\left[ \log_{(n)}\left(e^{e^{\iddots^{e}}} \ + \ K_{f}\right) \right]^{\beta} J_{f}\, \mathrm{d}z.
\end{equation}

By (\ref{only one}) and $J_{f^{-1}} \le |Df^{-1}|^2$, it follows from (\ref{main theorem 1}) that $J_{f^{-1}} \in L^{1}_{loc}(f(\Omega))$. So by Theorem 3.3 of \cite{hencl2006regularity}, we have $f^{-1}$ is a mapping of finite distortion.

Next we prove (\ref{main result}). Because of (\ref{main theorem 2}), it suffices to prove
\begin{equation}\label{main theorem 3}
\int_{M}\left[ \log_{(n)}\left(e^{e^{\iddots^{e}}} \ +\ K_{f}(z)\right) \right]^{\beta} J_{f}(z)\, \mathrm{d}z\ <\ \infty
\end{equation}
for any compact set $M \subset \Omega.$
Let
$$s \ =\ J_{f}(z)\ \ \ \mbox{ and }\ \ \ t\ =\ \left[\log_{(n)}\left(e^{e^{\iddots^{e}}}\ +\ K_{f}(z)\right)\right]^{\beta}.$$
Since
$$\mathcal{A}_{p,n}(\exp_{(n)}(t^{\frac{1}{\beta}}))\ \le \  \mathcal{A}_{p,n}(K_{f}(z)) \ + \ p\left(e^{e^{\iddots^{e}}}\ -\ 1\right),$$
it follows from Lemma \ref{Young} that there exist constants $C'$ and $C''$ such that
\begin{equation}\label{main theorem 3-1}
ts\ \le \ C' \exp[\mathcal{A}_{p,n}(K_{f})] \ + \ C''J_{f} \left[\log_{(n+1)}\left(e^{e^{\iddots^{e}}} \ +\ J_{f}\right)\right]^{\beta}.
\end{equation}
Note that $\mathcal{A}_{p,n}(x)$ satisfies (\ref{A-1}) and (\ref{A-2}) conditions, then Lemma \ref{thm20.5.1} implies
$$f \in W^{1,P_n}_{loc}(\Omega),$$
where $P_n$ is the associated function of $\mathcal{A}_{p,n}$. So, it follows from Lemma \ref{proposition} that
\begin{equation}\label{main theorem 4}
J_{f} \left[\log_{(n+1)}\left(e^{e^{\iddots^{e}}}\ +\ J_{f}\right)\right]^{\beta} \in L^{1}_{loc}(\Omega).
\end{equation}
Hence, according to (\ref{main theorem 3-1}), (\ref{main condition}) and (\ref{main theorem 4}), (\ref{main theorem 3}) is proved.

To show Theorem \ref{main theorem} is sharp, as in Theorem 4 of \cite{gill2010planar}, we consider Kovalev--type function $h$ in $\Omega\ =\ \mathbb{D}$ as
\begin{equation}\label{theorem 1-4}
h(z)\ =\ \frac{z}{|z|} \rho(|z|)
\end{equation}
where $\rho(t)\ =\ \left[ \log_{(n+1)}\left(e^{e^{\iddots^{e}}}\ +\ \frac{1}{t}\right)\right]^{-\frac{p}{2}} \left[ \log_{(n+2)}\left(e^{e^{\iddots^{e}}}\ +\ \frac{1}{t}\right)\right]^{-\frac{1}{2}}$ and both of $e^{e^{\iddots^{e}}}$ mean $\exp_{(n+1)}(e)$.
For the reader's convenience, we carry out the main computation.
By (\ref{main theorem 2}), it's enough to check
\begin{equation}\label{theorem 5-0}
J_{h} \left[ \log_{(n)}\left(e^{e^{\iddots^{e}}}+K_{h}\right) \right]^{p}\ \notin \ L_{loc}^1(\mathbb{D}).
\end{equation}
From the definition of $h$, it's sufficient to consider $h$ in the small enough neighbourhood of $0$.
So with the formulas in section 6.5.1 of \cite{iwaniec2001geometric}, when $|z|\ \ll \ 1$, we have
\begin{equation}\label{main theorem 5}
J_{h}(z)\ \sim \  \frac{1}{|z|^{2}} \frac{1}{\log_{(1)}(\frac{1}{|z|})}  \cdots   \frac{1}{\log_{(n)}(\frac{1}{|z|})} \left[\log_{(n+1)}(\frac{1}{|z|})\right]^{-p-1} \left[\log_{(n+2)}(\frac{1}{|z|})\right]^{-1}
\end{equation}
and
$$K_{h}(z) \ = \ \frac{\rho(|z|)}{|z|\rho'(|z|)}\ \sim\ \log_{(1)}(\frac{1}{|z|}) \log_{(2)}(\frac{1}{|z|}) \cdots \log_{(n+1)}(\frac{1}{|z|}).$$
Since
$$ \log \left(e^{e^{\iddots^{e}}} +K_{h}(z)\right)\ \sim \ \log(K_{h}(z))\ \sim \ \log_{(2)}(\frac{1}{|z|})\ \ \mbox{ as } \ |z| \ \ll \ 1,$$ we get
\begin{equation}\label{main theorem 6}
\left[\log_{(n)}\left(e^{e^{\iddots^{e}}} +K_{h}\right)\right]^{p}\ \sim\ \left[\log_{(n+1)}(\frac{1}{|z|})\right]^{p} \ \ \mbox{ as } \ |z| \ll 1.
\end{equation}
Combining (\ref{main theorem 5}) and (\ref{main theorem 6}), we obtain
\begin{align*}
J_{h} \left[ \log_{(n)}\left(e^{e^{\iddots^{e}}}+K_{h}\right) \right]^{p}\  \sim & \ \frac{1}{|z|^{2}} \frac{1}{\log_{(1)}(\frac{1}{|z|})} \cdots \frac{1}{\log_{(n+2)}(\frac{1}{|z|})}\\
\end{align*}
Now, (\ref{theorem 5-0}) is obtained from
\begin{align*}
 & \int_{0}^{*} \frac{1}{t} \frac{1}{\log_{(1)}(\frac{1}{t})}  \cdots  \frac{1}{\log_{(n+2)}(\frac{1}{t})} \, \mathrm{d}t\\
=\ & \int_{*}^{+\infty} \frac{1}{s} \frac{1}{\log_{(1)}(s)} \cdots \frac{1}{\log_{(n+2)}(s)} \, \mathrm{d}s \\
=\ &\cdots \ = \ \int_{*}^{+\infty} \frac{1}{\log(x)}\, \mathrm{d}x
\ =\  \infty.
\end{align*}
The proof is complete.
\end{proof}

\section*{Acknowledgements}
The author wishes to express his sincere appreciation to Professor Luigi Greco who provided the valuable method for Example \ref{example} and to his supervisor Professor Congwen Liu who critically read the manuscript and made numerous helpful suggestions.

\bigskip
\bibliographystyle{amsplain}

\end{document}